\documentclass[12pt,twoside]{amsart}
\usepackage{pifont}

\usepackage{amsmath, amsthm, amscd, amsfonts, amssymb, graphicx}
\usepackage{enumerate}
\usepackage[colorlinks=true,
linkcolor=magenta,
urlcolor=cyan,
citecolor=cyan]{hyperref}
\usepackage{mathrsfs}
\newtheorem{theorem}{Theorem}[section]

\newtheorem{remark}{Remark}[section]

\newtheorem{corollary}{Corollary}[section]
\newtheorem{example}{Example}[section]
\newtheorem{proposition}{Proposition}[section]
\numberwithin{equation}{section}

\begin{document}
\title{An extension of Jensen's operator inequality and its application  to Young inequality}
\author{Hamid Reza Moradi, Shigeru Furuichi, Flavia-Corina Mitroi-Symeonidis and Razieh Naseri}

\maketitle

\begin{abstract}
Jensen's operator inequality for convexifiable functions is obtained. This result contains classical Jensen's operator inequality  as a particular case. As a consequence, a new refinement and a reverse of Young's inequality are given.
\end{abstract}
\pagestyle{myheadings}
\markboth{\centerline {An extension of Jensen's operator inequality and its application  to Young inequality}}
{\centerline {H.R. Moradi, S. Furuichi, F.-C. Mitroi-Symeonidis \& R. Naseri}}
\bigskip
\bigskip
\section{\bf Introduction and Preliminaries}
In this article, $\mathcal{H}$ will denote a Hilbert space, and the term \lq\lq operator" we shall mean endormorphism of $\mathcal{H}$. The following result that provides an operator version for the Jensen inequality is due to Mond and Pe\v cari\'c \cite{7}: 
\begin{theorem}

	{\upshape(Jensen's operator inequality for convex functions).} Let $A\in \mathcal{B}\left( \mathcal{H} \right)$ be a self-adjoint operator with $Sp\left( A \right)\subseteq \left[ m,M \right]$ for some scalars $m<M$. If $f\left( t \right)$ is a convex function on $\left[ m,M \right]$, then  
	\begin{equation}\label{40}
	f\left( \left\langle Ax,x \right\rangle \right)\le \left\langle f\left( A \right)x,x \right\rangle ,
	\end{equation}
	for every unit vector $x\in \mathcal{H}$.
\end{theorem}
Over the years, various extensions and generalizations of \eqref{40} have been
obtained in the literature, e.g., \cite{Horvath, kian, moradi}. For this background we refer to any expository text such as \cite{1}.

The aim of this paper is to find an inequality which contains \eqref{40} as a special case. Our result also allows to obtain a refinement and a reverse for the scalar Young inequality. More precisely, it will be shown that for two non-negative numbers $a,b$ we have
\begin{equation*}
\begin{aligned}
{{K}^{r}}\left( h,2 \right)\exp \left( \left( \frac{v\left( 1-v \right)}{2}-\frac{r}{4} \right){{\left( \frac{a-b}{D} \right)}^{2}} \right)&\le \frac{a{{\nabla }_{v}}b}{a{{\sharp}_{v}}b} \\ 
& \le {{K}^{R}}\left( h,2 \right)\exp \left( \left( \frac{v\left( 1-v \right)}{2}-\frac{R}{4} \right){{\left( \frac{a-b}{D} \right)}^{2}} \right), 
\end{aligned}
\end{equation*}
where $r=\min \left\{ v,1-v \right\}$, $R=\max \left\{ v,1-v \right\}$, $D=\max \left\{ a,b \right\}$ and $K(h,2) = \frac{(h+1)^2}{4h}$ is the Kantorovich constant with $h = \frac{b}{a}$. 

To make the text more self-contained we give a brief overview of convexifiable functions. Given a continuous $f:I\to \mathbb{R}$ defined on the compact interval $I\subset \mathbb{R}$, consider a function $\varphi :I\times \mathbb{R}\to \mathbb{R}$ defined by $\varphi \left( x,\alpha  \right)=f\left( x \right)-\frac{1}{2}\alpha {{x}^{2}}$. If $\varphi \left( x,\alpha  \right)$ is a convex function on $I$ for some $\alpha ={{\alpha }^{*}}$, then $\varphi \left( x,\alpha  \right)$ is called a convexification of $f$ and ${{\alpha }^{*}}$ a convexifier on $I$. A function $f$ is convexifiable if it has a convexification.
It is noted in \cite[Corollary 2.9]{2} that if the continuously differentiable function $f$ has Lipschitz derivative (i.e., $\left| f'\left( x \right)-f'\left( y \right) \right|\le L\left| x-y \right|$ for any $x,y\in I$ and some constant $L$), then $\alpha =-L$ is a convexifier of $f$. 

The following fact concerning convexifiable functions plays an  important role in our discussion (see {{\cite[Corollary 2.8]{2}}}):
\begin{equation}\label{p} \tag{\bf P}
\text{If }f\text{ is twice continuously differentiable, then }\alpha =\underset{t\in I}{\mathop{\min }}\,f''\left( t \right)\text{ is a convexifier of $f$}.
\end{equation}
The reader may consult \cite{10} for additional information about this topic. For all
other notions used in the paper, we refer the reader to the monograph \cite{1}.
\section{\bf Main Results}
After the above preparation, we are ready to prove the analogue of \eqref{40} for non-convex functions. 
\begin{theorem}\label{6}
	{\upshape(Jensen's operator inequality for non-convex functions).} Let $f$ be a continuous convexifiable function on the interval $I$ and $\alpha $ a convexifier of $f$. Then 
	\begin{equation}\label{1}
	f\left( \left\langle Ax,x \right\rangle  \right)\le \left\langle f\left( A \right)x,x \right\rangle -\frac{1}{2}\alpha \left( \left\langle {{A}^{2}}x,x \right\rangle -{{\left\langle Ax,x \right\rangle }^{2}} \right),
	\end{equation}
	for every self-adjoint operator $A$ with $Sp\left( A \right)\subseteq I$ and every unit vector $x\in \mathcal{H}$.
\end{theorem}
\begin{proof}
	The idea of proof evolves from the approach in \cite{3}. Let  ${{g}_{\alpha }}:I\to \mathbb{R}$ with ${{g}_{\alpha }}\left( x \right)=\varphi \left( x,\alpha  \right)$.  According to the assumption, ${{g}_{\alpha }}\left( x \right)$ is convex. Therefore
	\begin{equation*}
	{{g}_{\alpha }}\left( \left\langle Ax,x \right\rangle  \right)\le \left\langle {{g}_{\alpha }}\left( A \right)x,x \right\rangle ,
	\end{equation*}
	for every unit vector $x\in \mathcal{H}$.  This expression is equivalent to the desired inequality \eqref{1}.
\end{proof}

\medskip

A few remarks concerning Theorem \ref{6} are in order.
\begin{remark}
	\hfill
	\begin{itemize}
		\item[(a)] Using the fact that for a convex function $f$ one can choose the convexifier $\alpha =0$,  one recovers the inequality \eqref{40}.
		\item[(b)] For continuously differentiable function $f$ with Lipschitz derivative and Lipschitz constant $L$, we have
		
		\[f\left( \left\langle Ax,x \right\rangle  \right)\le \left\langle f\left( A \right)x,x \right\rangle +\frac{1}{2}L\left( \left\langle {{A}^{2}}x,x \right\rangle -{{\left\langle Ax,x \right\rangle }^{2}} \right).\] 
	\end{itemize}
\end{remark}

\medskip

An important special case of Theorem \ref{6}, which refines inequality \eqref{40} can be explicitly
stated using the property \eqref{p}.
\begin{remark}\label{7}
	Let $f:I\to \mathbb{R}$ be a twice continuously differentiable strictly convex function and $\alpha =\underset{t\in I}{\mathop{\min }}\,f''\left( t \right)$. Then
	\begin{equation}\label{45}
	f\left( \left\langle Ax,x \right\rangle  \right)\le \left\langle f\left( A \right)x,x \right\rangle -\frac{1}{2}\alpha \left( \left\langle {{A}^{2}}x,x \right\rangle -{{\left\langle Ax,x \right\rangle }^{2}} \right)\le \left\langle f\left( A \right)x,x \right\rangle ,
	\end{equation}
	for every positive operator $A$ with $Sp\left( A \right)\subseteq I$ and every unit vector $x\in \mathcal{H}$.
\end{remark}
The inequality \eqref{45} is obtained in the paper \cite[Theorem 3.3]{moradi} (where this result was derived for the strongly convex functions) with a different technique (see also \cite{dragomir1}).
\medskip

The proof of the following corollary is adapted from the one of {{\cite[Theorem 1.3]{1}}}, but we put a sketch of the proof for the reader.
\begin{corollary}\label{4}
	Let $f$ be a continuous convexifiable function on the interval $I$ and $\alpha $ a convexifier. Let ${{A}_{1}},\ldots ,{{A}_{n}}$ be self-adjoint operators on  $\mathcal{H}$ with $Sp\left( {{A}_{i}} \right)\subseteq I$ for $1\le i\le n$ and ${{x}_{1}},\ldots ,{{x}_{n}}\in \mathcal{H}$ be such that $\sum\nolimits_{i=1}^{n}{{{\left\| {{x}_{i}} \right\|}^{2}}}=1$. Then
	\begin{equation}\label{5}
	f\left( \sum\limits_{i=1}^{n}{\left\langle {{A}_{i}}{{x}_{i}},{{x}_{i}} \right\rangle } \right)\le \sum\limits_{i=1}^{n}{\left\langle f\left( {{A}_{i}} \right){{x}_{i}},{{x}_{i}} \right\rangle }-\frac{1}{2}\alpha \left( \sum\limits_{i=1}^{n}{\left\langle A_{i}^{2}{{x}_{i}},{{x}_{i}} \right\rangle }-{{\left( \sum\limits_{i=1}^{n}{\left\langle {{A}_{i}}{{x}_{i}},{{x}_{i}} \right\rangle } \right)}^{2}} \right).
	\end{equation} 
\end{corollary}
\begin{proof}
In fact, $\mathrm{x}:=\left( \begin{matrix}
{{x}_{1}}  \\
\vdots   \\
{{x}_{n}}  \\
\end{matrix} \right)$ is a unit vector in the Hilbert space ${\mathcal{H}^n}$. If  we introduce the \lq\lq diagonal'' operator on ${\mathcal{H}^n}$
\[\mathrm{A}:=\left( \begin{matrix}
{{A}_{1}} & \cdots  & 0  \\
\vdots  & \ddots  & \vdots   \\
0 & \cdots  & {{A}_{n}}  \\
\end{matrix} \right),\] 
then, obviously, $Sp\left( \mathrm{A} \right)\subseteq I$, $\left\| \mathrm{x} \right\|=1$, $\left\langle f\left( \mathrm{A} \right)\mathrm{x},\mathrm{x} \right\rangle =\sum\nolimits_{i=1}^{n}{\left\langle f\left( {{A}_{i}} \right){{x}_{i}},{{x}_{i}} \right\rangle }$, $\left\langle \mathrm{A}\mathrm{x},\mathrm{x} \right\rangle =\sum\nolimits_{i=1}^{n}{\left\langle {{A}_{i}}{{x}_{i}},{{x}_{i}} \right\rangle }$, $\left\langle {{\mathrm{A}}^{2}}\mathrm{x},\mathrm{x} \right\rangle =\sum\nolimits_{i=1}^{n}{\left\langle A_{i}^{2}{{x}_{i}},{{x}_{i}} \right\rangle }$. Hence, to complete the proof, it is enough to apply Theorem \ref{6} for $\mathrm{A}$ and $\mathrm{x}$.
\end{proof}

\medskip

Corollary \ref{4} leads us to the following result. The argument depends on an idea of {{\cite[Corollary 1]{agarwal}}}.
\begin{corollary}\label{8}
	Let $f$ be a continuous convexifiable function on the interval $I$ and $\alpha $ a convexifier. Let ${{A}_{1}},\ldots ,{{A}_{n}}$ be self-adjoint operators on  $\mathcal{H}$ with $Sp\left( {{A}_{i}} \right)\subseteq I$ for $1\le i\le n$ and let ${{p}_{1}},\ldots ,{{p}_{n}}$ be positive scalars such that $\sum\nolimits_{i=1}^{n}{{{p}_{i}}}=1$. Then
	\begin{equation}\label{2}
	f\left( \sum\limits_{i=1}^{n}{\left\langle {{p}_{i}}{{A}_{i}}x,x \right\rangle } \right)\le \sum\limits_{i=1}^{n}{\left\langle {{p}_{i}}f\left( {{A}_{i}} \right)x,x \right\rangle }-\frac{1}{2}\alpha \left( \sum\limits_{i=1}^{n}{\left\langle {{p}_{i}}A_{i}^{2}x,x \right\rangle }-{{\left( \sum\limits_{i=1}^{n}{\left\langle {{p}_{i}}{{A}_{i}}x,x \right\rangle } \right)}^{2}} \right),
	\end{equation}
	for every unit vector $x\in \mathcal{H}$.
\end{corollary}
\begin{proof}
	Suppose that $x\in \mathcal{H}$ is a unit vector. Putting ${{x}_{i}}=\sqrt{{{p}_{i}}}x\in \mathcal{H}$ so that $\sum\nolimits_{i=1}^{n}{{{\left\| {{x}_{i}} \right\|}^{2}}}=1$ and applying Corollary \ref{4} we obtain the desired result \eqref{2}.
\end{proof}

\medskip

The clear advantage of our approach over the Jensen operator inequality is shown in the following example. Before proceeding we recall the following multiple operator version of Jensen's inequality \cite[Corollary 1]{agarwal}: Let $f:\left[ m,M \right]\subseteq \mathbb{R}\to \mathbb{R}$ be a convex function and ${{A}_{i}}$ be self-adjoint operators with $Sp\left( {{A}_{i}} \right)\subseteq \left[ m,M \right]$, $i=1,\ldots ,n$ for some scalars $m<M$. If ${{p}_{i}}\ge 0$, $i=1,\ldots ,n$ with $\sum\nolimits_{i=1}^{n}{{{p}_{i}}}=1$, then 
\begin{equation}\label{41}
f\left( \sum\limits_{i=1}^{n}{\left\langle {{p}_{i}}{{A}_{i}}x,x \right\rangle } \right)\le \sum\limits_{i=1}^{n}{\left\langle {{p}_{i}}f\left( {{A}_{i}} \right)x,x \right\rangle },
\end{equation}
for every $x\in \mathcal{H}$ with $\left\| x \right\|=1$.
\begin{example}
	We use the same idea from \cite[Illustration 1]{3}. Let $f\left( t \right)=\sin t\text{ }\left( 0\le t\le 2\pi  \right)$, $\alpha =\underset{0\le t\le 2\pi }{\mathop{\min }}\,f''\left( t \right)=-1$, $n=2$, ${{p}_{1}}=p$, ${{p}_{2}}=1-p$, $\mathcal{H}={{\mathbb{R}}^{2}}$, ${{A}_{1}}=\left( \begin{matrix}
	2\pi  & 0  \\
	0 & 0  \\
	\end{matrix} \right)$, ${{A}_{2}}=\left( \begin{matrix}
	0 & 0  \\
	0 & 2\pi   \\
	\end{matrix} \right)$
	and $x=\left( \begin{matrix}
	0  \\
	1  \\
	\end{matrix} \right)$. After simple calculations (thanks to the continuous functional calculus), from  \eqref{2} we infer that
	\begin{equation}\label{44}
	\sin \left( 2\pi \left( 1-p \right) \right)\le 2{{\pi }^{2}}p\left( 1-p \right),\qquad \text{ }0\le p\le 1
	\end{equation}
	and \eqref{41} implies 
	\begin{equation}\label{43}
	\sin \left( 2\pi \left( 1-p \right) \right)\le 0,\qquad \text{ }0\le p\le 1.
	\end{equation}
	Not so surprisingly, the inequality \eqref{43} can break down when $\frac{1}{2}\le p\le 1$ (i.e., \eqref{41} is not applicable here). However, the new upper bound in \eqref{44} holds.
\end{example}

\medskip

The weighted version of {{\cite[Theorem 3]{3}}} follows from Corollary \ref{8}, i.e.,
\begin{equation}\label{11}
f\left( \sum\limits_{i=1}^{n}{{{p}_{i}}{{t}_{i}}} \right)\le \sum\limits_{i=1}^{n}{{{p}_{i}}f\left( {{t}_{i}} \right)}-\frac{1}{2}\alpha \left( \sum\limits_{i=1}^{n}{{{p}_{i}}t_{i}^{2}}-{{\left( \sum\limits_{i=1}^{n}{{{p}_{i}}{{t}_{i}}} \right)}^{2}} \right),
\end{equation}
where ${{t}_{i}}\in I$ and $\sum\nolimits_{i=1}^{n}{{{p}_{i}}}=1$. For the case $n=2$, the inequality \eqref{11} reduces to
\begin{equation}\label{10}
f\left( \left( 1-v \right){{t}_{1}}+v{{t}_{2}} \right)\le \left( 1-v \right)f\left( {{t}_{1}} \right)+vf\left( {{t}_{2}} \right)-\frac{v\left( 1-v \right)}{2}\alpha {{\left( {{t}_{1}}-{{t}_{2}} \right)}^{2}},
\end{equation}
where $0\le v\le 1$. In particular
\begin{equation}\label{19}
f\left( \frac{{{t}_{1}}+{{t}_{2}}}{2} \right)\le \frac{f\left( {{t}_{1}} \right)+f\left( {{t}_{2}} \right)}{2}-\frac{1}{8}\alpha {{\left( {{t}_{1}}-{{t}_{2}} \right)}^{2}}.
\end{equation}
It is notable that Theorem \ref{6} is equivalent to the inequality (\ref{11}).
The following provides a refinement of the arithmetic-geometric mean inequality.
\begin{proposition}\label{proposition}
	For each $a,b>0$ and $0\le v\le 1$, we have
	\begin{equation}\label{14}
	\sqrt{ab}\le {{H}_{v}}\left( a,b \right)-\frac{d}{8}{{\left( \left( 1-2v \right)\left( \log \frac{a}{b} \right) \right)}^{2}}\le \frac{a+b}{2}-\frac{d}{8}{{\left( \log \frac{a}{b} \right)}^{2}}\le \frac{a+b}{2},
	\end{equation}
	where $d=\min \left\{ a,b \right\}$ and ${{H}_{v}}\left( a,b \right)= \frac{{{a}^{1-v}}{{b}^{v}}+{{b}^{1-v}}{{a}^{v}}}{2}$ is the Heinz mean.
\end{proposition}
\begin{proof}
	Assume that $f$ is a twice differentiable convex function such that $\alpha \le f''$ where $\alpha \in \mathbb{R}$. Under these conditions, it follows that
	\begin{equation*}
	\begin{aligned}
	f\left( \frac{a+b}{2} \right)&=f\left( \frac{\left( 1-v \right)a+vb+\left( 1-v \right)b+va}{2} \right) \\ 
	& \le \frac{f\left( \left( 1-v \right)a+vb \right)+f\left( \left( 1-v \right)b+va \right)}{2}-\frac{1}{8}\alpha {{\left( \left( a-b \right)\left( 1-2v \right) \right)}^{2}} \quad \text{(by \eqref{19})}\\ 
	& \le \frac{f\left( a \right)+f\left( b \right)}{2}-\frac{1}{8}\alpha {{\left( a-b \right)}^{2}} \quad \text{(by \eqref{10})}\\ 
	& \le \frac{f\left( a \right)+f\left( b \right)}{2},  
	\end{aligned}
	\end{equation*}
	for $\alpha \ge 0$. Now taking $f\left( t \right)={{e}^{t}}$ with $t \in I=\left[ a,b \right]$ in the above inequalities, we deduce the desired inequality \eqref{14}.
\end{proof}
\begin{remark}
	As Bhatia pointed out in \cite{bhatia}, the Heinz means interpolate between the geometric mean and the arithmetic mean, i.e.,
	\begin{equation}\label{15}
	\sqrt{ab}\le {{H}_{v}}\left( a,b \right)\le \frac{a+b}{2}.
	\end{equation}
	Of course, the first inequality in \eqref{14} yields an improvement of \eqref{15}. The inequalities in \eqref{14} also sharpens up the following inequality which is due to Dragomir (see \cite[Remark 1]{Dragomir}):
	\[\frac{d}{8}{{\left( \log \frac{a}{b} \right)}^{2}}\le \frac{a+b}{2}-\sqrt{ab}.\]
\end{remark}

\medskip

Studying about the arithmetic-geometric mean inequality, we cannot avoid mentioning its cousin, the Young inequality. The following inequalities provides a multiplicative type refinement and reverse of the Young's inequality:  
\begin{equation}\label{young}
{{K}^{r}}\left( h,2 \right)\le \frac{\left( 1-v \right)a+vb}{{{a}^{1-v}}{{b}^{v}}}\le {{K}^{R}}\left( h,2 \right),
\end{equation}
where $0\le v\le 1$, $r=\min \left\{ v,1-v \right\}$, $R=\max \left\{ v,1-v \right\}$ and $K(h,2) = \frac{(h+1)^2}{4h}$ with $h = \frac{b}{a}$. The first one was proved by Zuo et al. {{\cite[Corollary 3]{fujji}}}, while the second one was given by Liao et al. {{\cite[Corollary 2.2]{liao}}}.
	
Our aim in the following is to establish a refinement for the inequalities in \eqref{young}. The crucial role for our purposes will play the following facts:

If $f$ is a convex function on the fixed closed interval $I$, then   
\begin{equation}\label{16}
n\lambda \left\{ \sum\limits_{i=1}^{n}{\frac{1}{n}f\left( {{x}_{i}} \right)-f\left( \sum\limits_{i=1}^{n}{\frac{1}{n}{{x}_{i}}} \right)} \right\}\le \sum\limits_{i=1}^{n}{{{p}_{i}}f\left( {{x}_{i}} \right)}-f\left( \sum\limits_{i=1}^{n}{{{p}_{i}}{{x}_{i}}} \right),
\end{equation}
\begin{equation}\label{17}
\sum\limits_{i=1}^{n}{{{p}_{i}}f\left( {{x}_{i}} \right)}-f\left( \sum\limits_{i=1}^{n}{{{p}_{i}}{{x}_{i}}} \right)\le n\mu \left\{ \sum\limits_{i=1}^{n}{\frac{1}{n}f\left( {{x}_{i}} \right)-f\left( \sum\limits_{i=1}^{n}{\frac{1}{n}{{x}_{i}}} \right)} \right\},
\end{equation}
where ${{p}_{1}},\ldots ,{{p}_{n}}\ge 0$ with $\sum\nolimits_{i=1}^{n}{{{p}_{i}}}=1$, $\lambda =\min \left\{ {{p}_{1}},\ldots ,{{p}_{n}} \right\}$, $\mu =\max \left\{ {{p}_{1}},\ldots ,{{p}_{n}} \right\}$. Notice that the first inequality goes back to Pe\v cari\'c et al. {{\cite[Theorem 1, P.717]{mit}}}, while the second one was obtained by Mitroi in {{\cite[Corollary 3.1]{mitroi}}}.

\medskip

Now we come to the announced theorem. In order to simplify the notations, we put $a{{\sharp}_{v}}b={{a}^{1-v}}{{b}^{v}}$ and $a{{\nabla }_{v}}b=\left( 1-v \right)a+vb$.
\begin{theorem}\label{b}
	Let $a,b>0$ and $0\le v\le 1$. Then
	\begin{equation}\label{18}
	\begin{aligned}
	{{K}^{r}}\left( h,2 \right)\exp \left( \left( \frac{v\left( 1-v \right)}{2}-\frac{r}{4} \right){{\left( \frac{a-b}{D} \right)}^{2}} \right)&\le \frac{a{{\nabla }_{v}}b}{a{{\sharp}_{v}}b} \\ 
	& \le {{K}^{R}}\left( h,2 \right)\exp \left( \left( \frac{v\left( 1-v \right)}{2}-\frac{R}{4} \right){{\left( \frac{a-b}{D} \right)}^{2}} \right), 
	\end{aligned}
	\end{equation}
where $r=\min \left\{ v,1-v \right\}$, $R=\max \left\{ v,1-v \right\}$, $D=\max \left\{ a,b \right\}$ and $K(h,2) = \frac{(h+1)^2}{4h}$ with $h = \frac{b}{a}$.    
\end{theorem}
\begin{proof}
	Employing the inequality \eqref{16}  for the twice differentiable convex function $f$ with $\alpha \le f''$, we have
\[\begin{aligned}
& n\lambda \left\{ \frac{1}{n}\sum\limits_{i=1}^{n}{f\left( {{x}_{i}} \right)}-f\left( \frac{1}{n}\sum\limits_{i=1}^{n}{{{x}_{i}}} \right) \right\}-\sum\limits_{i=1}^{n}{{{p}_{i}}f\left( {{x}_{i}} \right)}+f\left( \sum\limits_{i=1}^{n}{{{p}_{i}}{{x}_{i}}} \right) \\ 
& \le \frac{\alpha }{2}\left\{ n\lambda \left[  \frac{1}{n}\sum\limits_{i=1}^{n}{x_{i}^{2}}-{{\left( \frac{1}{n}\sum\limits_{i=1}^{n}{{{x}_{i}}} \right)}^{2}} \right]-\left( \sum\limits_{i=1}^{n}{{{p}_{i}}x_{i}^{2}}-{{\left( \sum\limits_{i=1}^{n}{{{p}_{i}}{{x}_{i}}} \right)}^{2}} \right) \right\}. \\ 
\end{aligned}\]
	Here we set $n=2$,  $x_1=a$, $x_2=b$, $p_1=1-v$, $p_2=v$, $\lambda =r$ and $f(x) =-\log x$ with $I=[a,b]$ (so $\alpha =\underset{x\in I}{\mathop{\min }}\,f''\left( x \right)=\frac{1}{{{D}^{2}}}$). Thus we deduce the first inequality in \eqref{18}.
	The second  inequality in \eqref{18} is also obtained similarly by using the inequality \eqref{17}.
\end{proof}
\begin{remark}
	\hfill
	\begin{itemize}
		\item[(a)] Since $\frac{v(1-v)}{2}-\frac{r}{4}\ge 0$ for each $0\le v\le 1$, we have $\exp \left( {\left( {\frac{{v\left( {1 - v} \right)}}{2} - \frac{r}{4}} \right){{\left( {\frac{{a - b}}{D}} \right)}^2}} \right)\ge 1$. Therefore the first inequality in \eqref{18} provides an improvement for the first inequality in \eqref{young}.
		\vskip 0.2 true cm
		\item[(b)] Since $\frac{v(1-v)}{2}-\frac{R}{4}\le 0$ for each $0\le v\le 1$, we get $\exp \left( {\left( {\frac{{v\left( {1 - v} \right)}}{2} - \frac{R}{4}} \right){{\left( {\frac{{a - b}}{D}} \right)}^2}} \right)\le 1$. Therefore the second inequality in \eqref{18} provides an improvement for the second inequality in \eqref{young}.
	\end{itemize}
\end{remark}
\begin{proposition} \label{relation_lower_dragomir}
	Under the same assumptions in Theorem \ref{b}, we have
	$$
	\frac{(h+1)^2}{4h} \geq \exp\left( \frac{1}{4}\left(\frac{a-b}{D}\right)^2\right).
	$$
\end{proposition}
\begin{proof}
	We prove the case $a \leq b$, then $h \geq 1$. We set
	$
	f_1(h) \equiv 2\log(h+1)-\log h-2 \log 2 -\frac{1}{4}\frac{(h-1)^2}{h^2}.
	$ 
	It is quite easy to see that
	$
	f_1'(h)=\frac{(2h+1)(h-1)^2}{2h^3(h+1)} \geq 0,
	$
	so that $f_1(h) \geq f_1(1)=0$.
	For the case $a \geq b$, (then $0<h \leq 1$), we also set
	$
	f_2(h) \equiv  2\log(h+1)-\log h-2 \log 2 -\frac{1}{4}(h-1)^2.
	$
	By direct calculation
	$
	f_2'(h)=-\frac{(h-1)^2(h+2)}{2h(h+1)} \leq 0,
	$
	so that $f_2(h) \geq f_2(1) =0$. Thus the statement follows.
\end{proof}
\begin{remark}
	Dragomir obtained a refinement and reverse of Young's inequality in \cite[Theorem 3]{Dragomir} as:
	\begin{equation} \label{ineq_dragomir}
	\exp \left( {\frac{{v\left( {1 - v} \right)}}{2}{{\left( {\frac{{a - b}}{D}} \right)}^2}} \right) \le \frac{{a{\nabla _v}b}}{{a{\sharp _v}b}} \le \exp \left( {\frac{{v\left( {1 - v} \right)}}{2}{{\left( {\frac{{a - b}}{d}} \right)}^2}} \right),
	\end{equation}
	where $d = \min\{a,b\}$. From the following facts (a) and (b), we claim that our inequalities are non-trivial results.
	\begin{itemize}
		\item[(a)] From Proposition \ref{relation_lower_dragomir}, our lower bound in \eqref{18} is tighter than the one in (\ref{ineq_dragomir}).
		\item[(b)] Numerical computations show that there is no ordering between the right hand side in \eqref{18} and the one in the second inequality of (\ref{ineq_dragomir}) shown in \cite[Theorem 3]{Dragomir}. For example, if we take $a=2$, $b=1$ and $v=0.1$, then 
		$$
		{{K}^{R}}\left( h,2 \right)\exp \left( \left( \frac{v\left( 1-v \right)}{2}-\frac{R}{4} \right){{\left( \frac{a-b}{D} \right)}^{2}} \right)-\exp\left(\frac{v(1-v)}{2}\left(\frac{a-b}{d}\right)^2 \right)\simeq 0.0168761,
		$$
		whereas it approximately equals $-0.0436069$ when $a=2$, $b=1$ and $v=0.3$.
	\end{itemize}
\end{remark}

\medskip

We give a further remark in relation to comparisons with other inequalities.
\begin{remark}
	The following refined Young inequality and its reverse are known
	\begin{equation} \label{WZL_ineq}
	K^{r'}(\sqrt{t},2)t^v +r(1-\sqrt{t})^2 \leq (1-v) +v t \leq K^{R'}(\sqrt{t},2)t^v +r(1-\sqrt{t})^2,
	\end{equation}
	where $t>0$, $r'=\min\{2r,1-2r\}$ and $R'=\max\{2r,1-2r\}$. The first  and  second inequality were given in \cite[Lemma 2.1]{WZ2014} and in \cite[Theorem 2.1]{liao}, respectively.
	
	Numerical computations show that there is no ordering between our inequalities \eqref{18} and the above ones.
	Actually, if we take $v=0.45$ and $t=0.1$ (we set $t=\frac{b}{a}$ with $a \geq b$ in \eqref{18}), then 
	$$
	K^{R'}(\sqrt{t},2)t^v +r(1-\sqrt{t})^2 - t^v K^R(h,2) \exp\left( \left(\frac{v(1-v)}{2}-\frac{R}{4}\right)(1-t)^2 \right) \simeq 0.0363059,
	$$
	while it equals approximately $-0.0860004$ when $v=0.9$ and $t=0.1$. 
	
	Similarly, when $v=0.45$ and $t=0.1$ we get
	$$
	K^{r'}(\sqrt{t},2)t^v +r(1-\sqrt{t})^2- t^v K^r(h,2) \exp\left( \left(\frac{v(1-v)}{2}-\frac{r}{4}\right)(1-t)^2 \right) \simeq -0.0126828,
	$$
	while it equals approximately $0.037896$ when $v=0.9$ and $t=0.1$. 
\end{remark}

\medskip

Obviously, in the inequality \eqref{young}, we cannot replace ${{K}^{r}}\left( h,2 \right)$ by ${{K}^{R}}\left( h,2 \right)$, or vice versa. In this regard, we have the following theorem. The proof is almost the same as that of Theorem \ref{b} (it is enough to use the convexity of the function ${{g}_{\beta }}\left( x \right)=\frac{\beta }{2}{{x}^{2}}-f\left( x \right)$ where $\beta =\underset{x\in I}{\mathop{\max }}\,f''\left( x \right)$).  
\begin{theorem}\label{c}
Let all the assumptions of Theorem \ref{b} hold except that $d=\min \left\{ a,b \right\}$. Then
\[\begin{aligned}
{{K}^{R}}\left( h,2 \right)\exp \left( \left( \frac{v\left( 1-v \right)}{2}-\frac{R}{4} \right){{\left( \frac{a-b}{d} \right)}^{2}} \right)&\le \frac{a{{\nabla }_{v}}b}{a{{\sharp}_{v}}b} \\ 
& \le {{K}^{r}}\left( h,2 \right)\exp \left( \left( \frac{v\left( 1-v \right)}{2}-\frac{r}{4} \right){{\left( \frac{a-b}{d} \right)}^{2}} \right).  
\end{aligned}\]
 \end{theorem}

\medskip

We end this paper by presenting the operator inequalities based on Theorems \ref{b} and \ref{c}, thanks to the Kubo-Ando theory \cite{kubo-ando}.
\begin{corollary}
	Let $A$, $B$ be two positive invertible operators and positive real numbers $m$, $m'$, $M$, $M'$ that satisfy one of the following conditions:
	\begin{itemize}
		\item[(i)] $0<m'I\le A\le mI<MI\le B\le M'I$.
		\item[(ii)] $0<m'I\le B\le mI<MI\le A\le M'I$.
	\end{itemize}
	Then
	\begin{equation}\label{21}
	\begin{aligned}
	& {{K}^{r}}\left( h,2 \right)\exp \left( \left( \frac{v\left( 1-v \right)}{2}-\frac{r}{4} \right){{\left( \frac{1-h}{h} \right)}^{2}} \right)A{{\sharp}_{v}}B \\ 
	& \le A{{\nabla }_{v}}B \\ 
	& \le {{K}^{R}}\left( h',2 \right)\exp \left( \left( \frac{v\left( 1-v \right)}{2}-\frac{R}{4} \right){{\left( \frac{1-h'}{h'} \right)}^{2}} \right)A{{\sharp}_{v}}B
	\end{aligned}
	\end{equation}
	and
	\[\begin{aligned}
	& {{K}^{R}}\left( h,2 \right)\exp \left( \left( \frac{v\left( 1-v \right)}{2}-\frac{R}{4} \right){{\left( \frac{1-h'}{h'} \right)}^{2}} \right)A{{\sharp}_{v}}B \\ 
	& \le A{{\nabla }_{v}}B \\ 
	& \le {{K}^{r}}\left( h',2 \right)\exp \left( \left( \frac{v\left( 1-v \right)}{2}-\frac{r}{4} \right){{\left( \frac{1-h}{h} \right)}^{2}} \right)A{{\sharp}_{v}}B,  
	\end{aligned}\]
where $r=\min \left\{ v,1-v \right\}$, $R=\max \left\{ v,1-v \right\}$ and $K(h,2) = \frac{(h+1)^2}{4h}$ with $h=\frac{M}{m}$ and $h'=\frac{M'}{m'}$.
\end{corollary}
\begin{proof}
On account of \eqref{18}, we have
	\[\begin{aligned}
	& \underset{h\le x\le h'}{\mathop{\min }}\,\left\{ {{K}^{r}}\left( x,2 \right)\exp \left( \left( \frac{v\left( 1-v \right)}{2}-\frac{r}{4} \right){{\left( \frac{1-x}{\max \left\{ 1,x \right\}} \right)}^{2}} \right) \right\}{{T}^{v}} \\ 
	& \le \left( 1-v \right)I+vT \\ 
	& \le \underset{h\le x\le h'}{\mathop{\max }}\,\left\{ {{K}^{R}}\left( x,2 \right)\exp \left( \left( \frac{v\left( 1-v \right)}{2}-\frac{R}{4} \right){{\left( \frac{1-x}{\max \left\{ 1,x \right\}} \right)}^{2}} \right) \right\}{{T}^{v}},  
	\end{aligned}\]
	for the positive operator $T$ such that $hI\le T\le h'I$. Setting $T={{A}^{-\frac{1}{2}}}B{{A}^{-\frac{1}{2}}}$. 
	
	In the first case we have $I<hI=\frac{M}{m}I\le {{A}^{-\frac{1}{2}}}B{{A}^{-\frac{1}{2}}}\le \frac{M'}{m'}I=h'I$, which implies that
	\begin{equation}\label{20}
	\begin{aligned}
	& \underset{1\le h\le x\le h'}{\mathop{\min }}\,\left\{ {{K}^{r}}\left( x,2 \right)\exp \left( \left( \frac{v\left( 1-v \right)}{2}-\frac{r}{4} \right){{\left( \frac{1-x}{x} \right)}^{2}} \right) \right\}{{\left( {{A}^{-\frac{1}{2}}}B{{A}^{-\frac{1}{2}}} \right)}^{v}} \\ 
	& \le \left( 1-v \right)I+v{{A}^{-\frac{1}{2}}}B{{A}^{-\frac{1}{2}}} \\ 
	& \le \underset{1\le h\le x\le h'}{\mathop{\max }}\,\left\{ {{K}^{R}}\left( x,2 \right)\exp \left( \left( \frac{v\left( 1-v \right)}{2}-\frac{R}{4} \right){{\left( \frac{1-x}{x} \right)}^{2}} \right) \right\}{{\left( {{A}^{-\frac{1}{2}}}B{{A}^{-\frac{1}{2}}} \right)}^{v}}.  
	\end{aligned}
	\end{equation}
	We can write \eqref{20} in the form
	\[\begin{aligned}
	& {{K}^{r}}\left( h,2 \right)\exp \left( \left( \frac{v\left( 1-v \right)}{2}-\frac{r}{4} \right){{\left( \frac{1-h}{h} \right)}^{2}} \right){{\left( {{A}^{-\frac{1}{2}}}B{{A}^{-\frac{1}{2}}} \right)}^{v}} \\ 
	& \le \left( 1-v \right)I+v{{A}^{-\frac{1}{2}}}B{{A}^{-\frac{1}{2}}} \\ 
	& \le {{K}^{R}}\left( h',2 \right)\exp \left( \left( \frac{v\left( 1-v \right)}{2}-\frac{R}{4} \right){{\left( \frac{1-h'}{h'} \right)}^{2}} \right){{\left( {{A}^{-\frac{1}{2}}}B{{A}^{-\frac{1}{2}}} \right)}^{v}}.  
	\end{aligned}\]
	Finally, multiplying both sides of the previous inequality by ${{A}^{\frac{1}{2}}}$ we get the desired result \eqref{21}.
	
	The proof of other cases is similar, we omit the details. 
\end{proof}

\section*{Acknowledgement}
The authors thank anonymous referees for giving valuable comments and suggestions to improve our manuscript. 
The author (S.F.) was partially supported by JSPS KAKENHI Grant Number
16K05257.

\vskip 0.6 true cm

{\tiny (H.R. Moradi) Young Researchers and Elite Club, Mashhad Branch, Islamic Azad University, Mashhad, Iran.
	
{\it E-mail address:} hrmoradi@mshdiau.ac.ir
	
\vskip 0.4 true cm
	
{\tiny (S. Furuichi) Department of Information Science, College of Humanities and Sciences, Nihon University, 3-25-40, Sakurajyousui, Setagaya-ku, Tokyo, 156-8550, Japan.
		
{\it E-mail address:} furuichi@chs.nihon-u.ac.jp
		
\vskip 0.4 true cm

(F.-C. Mitroi-Symeonidis) Department of Mathematical Methods and Models, Faculty of Applied Sciences, University Politehnica of Bucharest, Romania.

{\it E-mail address:} fcmitroi@yahoo.com
	
\vskip 0.4 true cm
		
{\tiny (R. Naseri) Department of Mathematics, Payame Noor University, P.O. Box 19395-3697, Tehran, Iran.
	
{\it E-mail address:} raziyehnaseri29@gmail.com

\end{document}